\documentclass{amsart}
\usepackage{graphicx}
\usepackage{dsfont}
\usepackage{amssymb}
\usepackage{xcolor}
\usepackage{enumerate}
\newtheorem{proposition}{Proposition}
\newtheorem{example}[proposition]{Example}
\newtheorem{lemma}[proposition]{Lemma}
\newtheorem{corollary}[proposition]{Corollary}
\newtheorem{theorem}[proposition]{Theorem}

\newcommand{\set}[1]{\left\{#1\right\}}

\newcommand{\II}{\mathds{I}}

\newcommand{\RRR}{\overline{\mathds{R}}}
\newcommand{\DD}{\mathcal{D}}
\begin{document}

\title[Sklar's Theorem]{A full scale Sklar's theorem in the imprecise setting}%
\author{Matja\v{z} Omladi\v{c}}%
\address{Institute of Mathematics, Physics and Mechanics, Ljubljana, Slovenia}%
\email{Matjaz@Omladic.net}%
\author{Nik Stopar}%
\address{Faculty of electrical engineering, University of Ljubljana, and Institute of Mathematics, Physics and Mechanics, Ljubljana, Slovenia}%
\email{Nik.Stopar@fe.uni-lj.si}%

\thanks{The authors acknowledge financial support from the Slovenian Research Agency (research core funding No. P1-0222).}
\subjclass{AMS}%

\begin{abstract}
  In this paper we present a surprisingly general extension of the main result of a paper that appeared in this journal: I.\ Montes et al., \textsl{Sklar's theorem in an imprecise setting}, Fuzzy Sets and Systems, \textbf{278} (2015), 48--66. The main tools we develop in order to do so are: (1) a theory on quasi-distributions based on an idea presented in a paper by R.\ Nelsen with collaborators; (2) starting from what is called (bivariate) $p$-box in the above mentioned paper we propose some new techniques based on what we call restricted  (bivariate) $p$-box; and (3) a substantial extension of a theory on coherent imprecise copulas developed by M.\ Omladi\v{c} and N.\ Stopar in a previous paper in order to handle coherence of restricted (bivariate) $p$-boxes. A side result of ours of possibly even greater importance is the following: Every bivariate distribution whether obtained on a usual $\sigma$-additive probability space or on an additive space can be obtained as a copula of its margins meaning that its possible extraordinariness depends solely on its margins. This might indicate that copulas are a stronger probability concept than probability itself.
\end{abstract}
\maketitle

\section{Introduction}\label{sec:intro}

Dependence concepts play a crucial role in multivariate statistical literature since it was recognized that the independence assumption cannot describe conveniently the behavior of a stochastic system. One of the main tools in modeling these concepts have eventually become copulas due to their theoretical omnipotence emerging from the Sklar's theorem \cite{Skla} (see also the monographs \cite{AlFrSc,DuSe,Nels}). Namely, they are used to represent and construct joint distribution functions of random vectors in terms of the related one-dimensional marginal distribution functions. As a more general concept, quasi-copulas were introduced in \cite{AlNeSc} and an equivalent definition was given later in \cite{GeQuMoRoLaSe}. Quasi-copulas have interesting applications in several areas, such as fuzzy logic \cite{HaMe,SaTuMe}, fuzzy preference modeling \cite{DeScDeMeDeBa,DiMoDeBa} or similarity measures \cite{DeBaJaDeMe}. Other deep results concerning quasi-copulas can be found in \cite{DeBa,JaDeBaDeMe,NeQuMoRoLaUbFl}.

While copulas are characterized (in the bivariate case) by the nonnegativity of the volume of each sub-rectangle of the unit square $R\subseteq\II^2$ (here $\II=[0,1]$), this is no longer true for quasi-copulas. 
If $\mathcal{C}$ is any nonempty set of (quasi)copulas, then
\begin{equation}\label{coherent1}
  \underline{C}= \inf\{C\}_{C\in\mathcal{C}}\quad\mbox{and}\quad \overline{C}= \sup\{C\}_{C\in\mathcal{C}}
\end{equation}
(in the point-wise sense) are quasi-copulas by \cite[Theorem 6.2.5]{Nels}. In this respect the set of quasi-copulas is a complete lattice and may be actually seen as an order completion of the set of all copulas. The authors of \cite{MoMiPeVi,PeViMoMi1, PeViMoMi2} introduce their definition of an \emph{imprecise copula} as a set of axioms on a pair of quasi-copulas $(P,Q)$ (cf.\ Conditions (IC1)--(IC4) in Section \ref{sec:impr_sklar}) following the ideas of $p$-boxes and show that the pair $(\underline{C}, \overline{C})$ is always an imprecise copula ``representing'' the set of copulas $\mathcal{C}$ lying pointwise between the two bounds. Montes et al.\ \cite{MoMiPeVi} propose a question in the other direction whether every imprecise copula can be obtained in this way and Omladi\v{c} and Stopar \cite{OmSt} answer this question in the negative using and substantially improving the methods of Dibala et al.\ \cite{DiSaPlMeKl}. Following Omladi\v{c} and \v{S}kulj \cite{OmSk} we will call an imprecise copula $\mathcal{C}$ defined by a pair $(P,Q)$ \emph{coherent} if (with notation \eqref{coherent1}) we have $P=\underline{C}$ and $Q=\overline{C}$.

One of the results of Montes et al.\ \cite{MoMiPeVi} is an imprecise extension of Sklar's theorem but only in one direction using the definition of a bivariate $p$-box introduced in Pelessoni et al.\ \cite{PeViMoMi2}. It is the main goal of this paper to build a full scale imprecise Sklar's theorem; of course, in order to do that we need to adjust somewhat the notion of a bivariate $p$-box as well. Actually, the first attempt of an imprecise Sklar's theorem was given by Nelsen et al.\ \cite[Theorem 2.4]{NeQuMoRoLaUbFl} although they are not calling it so. Our Theorem \ref{thm:sklar_fix}, a Sklar's type theorem in the imprecise setting with fixed margins, may be seen as a substantial extension of their result. They also introduce briefly the notion of a quasi-distribution as a composition of a quasi-copula with given univariate margins. Here we give an axiomatic definition of a quasi-distribution in Section \ref{sec:quasi} and prove in Theorem \ref{thm:sklar_quasi} that our definition is equivalent to theirs. An important outcome of this theorem is that every bivariate distribution whether realized on a usual $\sigma$-additive probability space or on an additive space can be written as a copula in its margins meaning that its possible extraordinariness depends solely on its marginal distributions. Among the main results of this paper we should also point out two versions of a full scale Sklar's theorem in the imprecise setting, Theorems \ref{thm:sklar1} and \ref{thm:sklar2}.


The paper is organized as follows. Section \ref{sec:quasi} gives a novel approach to discrete distributions and presents a new approach to both general and discrete quasi-distributions; as already mentioned Theorem \ref{thm:sklar_quasi} proves equivalence of our definition with the definition of Nelsen et al.\ \cite{NeQuMoRoLaUbFl}. In Section \ref{sec:tools} we elaborate extensively the methods of \cite{OmSt} developed there as a tool for equivalent definition of coherence of an imprecise copula. The main result of this section is Theorem \ref{thm:sup inf} that states in principle the sufficient and necessary condition for a restricted bivariate $p$-box to be coherent. Section \ref{sec:impr_sklar} brings finally all these notions together and presents a few versions of Sklar's type theorem in the imprecise setting. An engaged reader may also find there interesting examples giving evidence why our approach may have some advantages over the previous ones.

\section{Discrete vs.\ general quasi-distributions }\label{sec:quasi}

In this section we will study functions of two variables whose domain is either $\mathds{D}=\overline{\mathds{R}}\times\overline{\mathds{R}}$, where $\overline{\mathds{R}}=\mathds{R}\cup\{-\infty,\infty\}$, or a mesh $\Delta=\delta_x\times \delta_y$ determined by some points
\[
    \delta_x=\{-\infty= x_0<x_1<\cdots<x_p=\infty\}\ \mbox{and}\ \delta_y=\{-\infty=y_0<y_1<\cdots<y_q=\infty\}.
\]
We will always assume that each of the sets $\delta_x,\delta_y,$ contains at least one strictly positive point different from $\infty$ and at least one strictly negative point different from $-\infty$. Let $\DD$ denote either $\mathds{D}$ or $\Delta$. A function defined on $\DD$ will be called \emph{general} if $\DD=\mathds{D}$ and \emph{discrete} if $\DD=\Delta$.

Now, choose a rectangle $R$ with \emph{standard corners} $\mathbf{a}, \mathbf{b}, \mathbf{c},$ and $\mathbf{d}$ (by this we mean that the rectangle is positively oriented and that $\mathbf{a}$ is the southwest corner). The sides of all our rectangles will be parallel to the axes. So, $R$ is determined by vertices $\mathbf{a}$ and $\mathbf{c}$ and will often be denoted by $[\mathbf{a},\mathbf{c}]$. For a mesh  $\Delta$ we will say that it is \emph{determined by the rectangles} $[({x}_{i-1},{y}_{j-1}),({x}_{i},{y}_{j})]$ for all $i\in[p]$ and $j\in[q]$. Here and in the sequel we denote by $[n]$ the set $\{1,2,\ldots,n\}$ for any given integer $n$.

Assume that the standard corners of rectangle $R$ are contained in $\DD$ and let $A$ be any function defined on $\DD$. We let the \emph{volume of $R$ with respect to} $A$, or the \emph{$A$-volume of} $R$ (or simply the volume of $R$ if $A$ is understood), be equal to
\[
    V_A(R)= A(\mathbf{a})+A(\mathbf{c})-A(\mathbf{b})-A(\mathbf{d}).
\]

Consider the following possible conditions on a function $F:\DD\to [0,1]$.

\begin{enumerate}[(A)]
	\item {\begin{enumerate}[(1)]
	         \item $F(x,-\infty)=F(-\infty,y)=0$ for all $x,y\in\overline{\mathds{R}}$ if $F$ is a general function and for all $x\in\delta_x$ and $y\in\delta_y$ if $F$ is a discrete function, and
	         \item $F(\infty,\infty)=1$.
	       \end{enumerate} }
	\item Condition (A) together with the following condition: $V_F(R)\geqslant0$ for every rectangle $R$ with corners in $\DD$ that intersects the boundary of $\DD$.
	\item Condition (A) together with the following condition: $V_F(R)\geqslant0$ for every rectangle $R$ with corners in $\DD$.
\end{enumerate}

In the discrete case we understand the boundary of $\DD$ in Condition (B) as those points that have at least one coordinate equal to either $-\infty$ or $\infty$.
Functions with property (B) will be called \emph{(bivariate) quasi-distributions}. Observe that quasi-distributions are always increasing in each variable.
Functions with property (C) will be called \emph{(bivariate) distributions}. In the general case we will sometimes add the word \emph{general} to either of these two notions and in the discrete case we will add the word \emph{discrete}, unless there will be no doubt of confusion.

Let us warn the reader not to confuse our notion of ``discrete distribution'' with the standard notion of the distribution of a discrete random vector, although they are connected in some sense. Namely, for a distribution $F$ of a discrete random vector with finite range one could choose to be defined only on the mesh generated by the points of discontinuity and one could recover the distribution uniquely from the values at these points.

Observe that this approach brings us to (cumulative) distributions (and quasi-distributions) of standard type (i.e., those that arise on the usual $\sigma$-additive probability spaces) if we assume in addition to either Condition (B) or (C) that:
\begin{enumerate}[(D)]
  \item The (quasi) distribution function $F$ is right continuous in each argument.
\end{enumerate}
We will also allow the possibility that (quasi)distributions do not satisfy this assumption which brings us to the case of finitely additive probability spaces. We will say that (quasi)distributions are \emph{of standard type} in the first case and that they are \emph{of extended type} in the second case. It will follow from our Theorem \ref{thm:sklar_quasi} that in both cases we can express a (quasi)distribution as a (quasi)copula in its marginal distribution functions. Consequently, possible extraordinariness of a (quasi)distribution depends solely on its margins! Does this mean that copulas are a stronger probability concept than probability itself?

Recall the definition of the bilinear interpolation from \cite{OmSt} on a mesh determined by rectangles of bounded corners. Let us extend this definition to unbounded corners using linear rational functions. If $\mathbf{a}$ has coordinates $(-\infty,b)$ and $\textbf{b}$ has coordinates $(a,b)$ then we extend the discrete function $A$ defined at these two points via $A(-\infty,b)$ and $A(a,b)$ by letting
\begin{equation}\label{eq:rational1}
  A(x,b)= A(-\infty,b)+\dfrac{a}{x}(A(a,b)-A(-\infty,b))\ \ \mbox{for}\ \ -\infty \leqslant x\leqslant a.
\end{equation}
Observe that the assumption from the beginning of this section yields $a<0$ so that the function $\dfrac{a}{x}$ is strictly increasing on the interval $(-\infty,a)$ and consequently the function $A(x,b)$ is increasing on this interval if and only if it is obtained from an increasing discrete function meaning that $A(a,b)\geqslant A(-\infty,b)$. Similar considerations apply when corner  $\mathbf{a}$ is bounded having coordinates $(a,b)$ and corner $\textbf{b}$ is unbounded having coordinates $(\infty,b)$. In this case we extend the discrete function $A$ by
\begin{equation*}\label{eq:rational2}
  A(x,b)= A(\infty,b)-\dfrac{a}{x}(A(\infty,b)-A(a,b))\ \ \mbox{for}\ \ a\leqslant x\leqslant \infty.
\end{equation*}
Since $a>0$ we have that $-\dfrac{a}{x}$ is strictly increasing on $(a,\infty)$ so that the function $A(x,b)$ is increasing in $x$ on this interval if and only if it is obtained from an increasing discrete function meaning that $A(\infty,b)\geqslant A(a,b)$. Finally, we extend the function $A$ along the vertical lines as well. Choose a rectangle $R$ determining the mesh $\Delta$ with standard notation of the corners $\mathbf{a}, \mathbf{b}, \mathbf{c},$ and $\mathbf{d}$. Observe that at least one of these corners is bounded. We will consider two cases. Case (a) when both corners of the horizontal sides are either bounded or both are unbounded; and case (b) otherwise. In case (a) we extend the function $A$ along these horizontal sides linearly and then extend vertically along each line with a fixed coordinate $x$ by analogy to the above. In case (b) we extend the function $A$ along the horizontal lines as above and then extend vertically along each line with a fixed coordinate $x$ either linearly or by analogy to the above. The result of \cite[Proposition 1]{OmSt} extends easily.

\begin{proposition}\label{bilinear}
  Given the values of a 1-increasing function $A$ at the corners of a rectangle $R$ there exists a unique function $A$ on $R$ such that
\begin{description}
  \item[(a)] its values coincide with the starting values at the corners;
  \item[(b)] each one-dimensional section parallel to the axes is either linear if both ends of the section are of the same bounded/unbounded type or it is linear rational with pole at $0$ otherwise.
\end{description}
\end{proposition}

The function on $R$ obtained in this way will be called a \textsl{bilinear interpolation} of $A$ through its values at the corners and denoted by $A^{\mathrm{BL}}$. This definition extends a function $A$ defined on a mesh with no ambiguity to a function denoted by $A^{\mathrm{BL}}$ on the whole unit square $\II^2$ and called again a \textsl{bilinear interpolation} of $A$.

\begin{corollary}\label{positive}
  The bilinear interpolation $A$ of a 1-increasing function defined on the corners of $R$ is 1-increasing. Moreover, for every subrectangle $R_1\subseteq R$ we have:
\begin{description}
  \item[(a)] $V_A(R_1)>0$ if and only if $V_A(R)>0$
  \item[(b)] $V_A(R_1)<0$ if and only if $V_A(R)<0$
  \item[(c)] $V_A(R_1)=0$ if and only if $V_A(R)=0$
\end{description}
\end{corollary}

\begin{proof}
  This was proven in \cite{OmSt} for rectangles with all the four corners bounded. Among rectangles with some unbounded corners we first consider the case that corners $\mathbf{a}= (-\infty,b)$ and $\mathbf{d}= (-\infty,c)$ of rectangle $R$ are unbounded and its corners $\mathbf{b}= (a,b)$ and $\mathbf{c}= (a,c)$ are bounded. If we change the east vertical side from coordinate $x=a$ westwards to coordinate $x=a'<a$ the volume changes from the starting volume
  \begin{equation*}\label{startvolume}
    V_A(R)=A(-\infty,b)-A(-\infty,c)+A(a,c)-A(a,b)
  \end{equation*}
  to $\dfrac{a}{a'} V_A(R)$ as seen after a simple computation using \eqref{eq:rational1}. If we change the north horizontal side from coordinate $y=c$ southwards to coordinate $y=c'<c$ the volume changes from the starting volume to $\dfrac{c'-b}{c-b}V_A(R)$ as seen via linear interpolation, and if  we change the south horizontal side from coordinate $y=b$ northwards to coordinate $y=b'>b$ the volume changes from the starting volume to $\dfrac{c-b'}{c-b}V_A(R)$ again seen via linear interpolation. Now, if we change the unbounded west vertical side to a bounded coordinate $x=d, -\infty<d<a<0$, the volume changes, using \eqref{eq:rational1} to $\dfrac{a-d}{-d}V_A(R)$. Consequently, in any of the four possible cases, the sign of the volume stays unchanged.

  The other possibilities are obtained using similar tedious but simple computations.
\end{proof}

\begin{proposition}
  For any mesh $\Delta$ and any of the three conditions above (X)=(A), (X)=(B), or (X)=(C), it holds that
  \begin{enumerate}
    \item If a general function $A:\mathds{D}\to \mathds{R}$ satisfies Condition (X), then the discrete function $A|_\Delta$ satisfies Condition (X).
    \item If a discrete function $A:\Delta\to\mathds{R}$ satisfies Condition (X), then its bilinear interpolation $A^\mathrm{BL}$ satisfies Condition (X).
  \end{enumerate}
\end{proposition}

\begin{proof}
  The proof goes in a similar way as the proof of \cite[Proposition 4]{OmSt}.
\end{proof}

Let $F_X(x)$ and $F_Y(y)$ be two univariate distributions either of standard or of extended type. Furthermore, let $\mathfrak{Q}^{F_X,F_Y}$, respectively $\mathfrak{D}^{F_X,F_Y}$, be the set of all quasi-distributions, respectively distributions, with these two univariate distributions as margins, i.e.\ such that
\[
    F(x,\infty)= F_X(x)\ \ \mbox{and}\ \ F(\infty,y)= F_Y(y).
\]

\begin{theorem}[Sklar's theorem for quasi-distributions] \label{thm:sklar_quasi}
  For any function $F$ in $\mathfrak{Q}^{F_X,F_Y}$, respectively $\mathfrak{D}^{F_X,F_Y}$, there exists a quasi-copula, respectively a copula, $C$ such that
  \[
    F(x,y)=C(F_X(x),F_Y(y)).
  \]
\end{theorem}

\begin{proof}
  It suffices to prove the quasi-copula part. The well-known copula part follows also easily from that. We adjust one of the standard proofs of Sklar's theorem (cf., say, \cite[Section 2.3]{DuSe}). There is at most a countable set of discontinuities of a monotone function (Froda's theorem) that are all jumps; denote the set of jumps of $F_X$ by $\{x_i\}$ and the set of jumps of $F_Y$ by $\{y_j\}$. (By an abuse of notation we count a jump twice in case that the distribution function is not continuous on either the left hand side or the right hand side to take care of both the left and the right jump; observe that this cannot occur in case of standard probability.)
  For $(u,v)\in \mathrm{Ran}\,F_X\times \mathrm{Ran}\,F_Y$ (here we denote by $\mathrm{Ran}\,F$ the range of a function $F$) there exist $x,y\in \overline{\mathds{R}}$ such that $u=F_X(x)$ and $v=F_Y(y)$; so we let $C(u,v)=F(x,y)$. It is easy to extend this definition by monotonicity to the set $\overline{\mathrm{Ran}\,F_X}\times \overline{\mathrm{Ran}\,F_Y}$, where the over-line denotes the closure as usually. Indeed, if $v=F_Y(y)$ and $u\in\overline{ \mathrm{Ran}\,F_X} \setminus\mathrm{Ran}\,F_X$, then either $u=\lim_{x\to x_i-0}F_X(x)$ or $u=\lim_{x\to x_i+0}F_X(x)$ for one of the jumps $x_i$, so that one should let $C(u,v)=\lim_{x\to x_i-0}F(x,y)$ in the first case and $C(u,v)=\lim_{x\to x_i+0}F(x,y)$ in the second one. Observe that both limits exist since every quasi-distribution is monotone in each of the two variables. In case of a double jump one needs to adjust this definition in an obvious way.
  Once $C$ is extended to $\overline{\mathrm{Ran}\,F_X}\times \mathrm{Ran}\,F_Y$ in this way, we extend it similarly to the whole desired set.
  So, we need to extend $C$ to two more types of sets:
  \begin{enumerate}[(a)]
    \item $J(x_i)\times\overline{\mathrm{Ran}\,F_Y}$, and $\overline{\mathrm{Ran}\,F_X}\times J(y_j)$;
    \item $J(x_i)\times J(y_j)$;
  \end{enumerate}
  where $J(x_i)$ and $J(y_j)$ are open intervals representing the jumps of the corresponding distribution functions. Observe that (only) the endpoints of $J(x_i)$ belong to $\overline{\mathrm{Ran}\,F_X}$ and (only) the endpoints of $J(y_j)$ belong to $\overline{\mathrm{Ran}\,F_Y}$. So, we can extend the value of $C$ linearly in one of the variables on the sets of type (a). Also, (only) the vertices of any rectangle of type (b) belong to $\overline{\mathrm{Ran}\,F_X}\times \overline{\mathrm{Ran}\,F_Y}$. Thus, we can extend the value of $C$ bilinearly on these sets.

  It remains to see that the so obtained function $C$ is a quasi-copula. We will prove this using an equivalent definition of \cite[p.\ 236]{Nels} (cf.\ also the definitions immediately following \cite[Proposition 3]{OmSt}). We see that $C$ satisfies Conditions (A) and (B) of \cite[Section 2]{OmSt} via a routine verification. So, we only need to show that Condition (D) of \cite[Section 2]{OmSt} (which is our Condition (B)) also holds true. Choose an arbitrary rectangle $R$ with standard corners $\mathbf{a}$, $\mathbf{b}$, $\mathbf{c}$, and $\mathbf{d}$, intersecting the boundary of $\mathds{I}^2$. If all the corners belong to $\mathrm{Ran}\,F_X \times \mathrm{Ran}\,F_Y$, then we can find a rectangle $R'$ with corners $\mathbf{a}'$, $\mathbf{b}'$, $\mathbf{c}'$, and $\mathbf{d}'$, intersecting the boundary of $\overline{\mathds{R}}^2$ and such that $\mathbf{a}=(F_X\times F_Y)(\mathbf{a}')$, $\mathbf{b}=(F_X\times F_Y)(\mathbf{b}')$, $\mathbf{c}=(F_X\times F_Y)(\mathbf{c}')$, and $\mathbf{d}=(F_X\times F_Y)(\mathbf{d}')$. It follows easily that $V_C(R) =V_F(R')\geqslant0$ by Condition (B) on the function $F$. So, the desired condition
  \begin{equation}\label{sklar+}
    V_C(R)\geqslant0
  \end{equation}
  is satisfied as soon as $R$ has nonempty intersection with the boundary of $\DD$ and the vertices of $R$ belong to $\mathrm{Ran}\,F_X \times \mathrm{Ran}\,F_Y$.
  In the rest of the cases the construction of $C$ 
  is done in two steps. First, we send the coordinates to the limits which clearly preserves Condition \eqref{sklar+}. Second, it is done either in linear or in bilinear way from vertices for whose rectangles Condition \eqref{sklar+} already holds (depending on whether they are of type (a) or type (b)). The value of $V_C(R)$ may be seen as either a linear or a bilinear function. Since a linear function cannot change the sign more than once on an interval and since $V_C(R)$ is positive at the endpoints, it is positive everywhere on the interval and the theorem follows.
\end{proof}

\begin{corollary}\label{prop:compact}
  The sets $\mathfrak{Q}^{F_X,F_Y}$ and $\mathfrak{D}^{F_X,F_Y}$ are compact in the uniform norm.
\end{corollary}

\begin{proof}
  This is an immediate consequence of Theorem \ref{thm:sklar_quasi} and Theorems \cite[1.7.7\,\&\,7.3.1]{DuSe}.
\end{proof}

\section{Introducing the main tools}\label{sec:tools}

Here we present some tools to treat copulas and quasi-copulas from \cite{OmSt} expanding them from the unit square to the extended real plane thus making them useful in treating distributions and quasi-distributions. Since many of the proofs are quite technical on one side and depend only on the relations between certain functions defined on rectangles (and disjoint unions of them), we give these results without proofs. We only exhibit the deeper proofs where it becomes important that the rectangles belong to $\RRR^2$ and not only to $\II^2$.

As in Section \ref{sec:quasi} we assume that a real function $A$ is defined on $\DD$ which is either $\RRR^2$ in the general case or a mesh $\Delta$ in the discrete case. For a rectangle $R$ with {distinct} standard vertices $\mathbf{a}, \mathbf{b}, \mathbf{c}, \mathbf{d}$ we define the \textsl{main corner set} $M(R)=\{\mathbf{a},\mathbf{c}\}$ and the \textsl{opposite corner set} $O(R)= \{\mathbf{b},\mathbf{d}\}$. Given a rectangle $R$ we define for any point $\mathbf{x}\in\RRR^2$ its multiplicity by
\[
    m_R(\mathbf{x})=\left\{
           \begin{array}{ll}
             1, & \hbox{if $\mathbf{x}\in M(R)$;} \\
             -1, & \hbox{if $\mathbf{x}\in O(R)$;} \\
             0, & \hbox{otherwise.}
           \end{array}
         \right.
\]
Let us amplify this definition to any $R\in\mathfrak{R}$, the set of all disjoint unions of rectangles with vertices in $\DD$, i.e.\ if $\{R_i\}_{i=1}^{n}$ is an arbitrary finite set of rectangles of the kind, then an element of $\mathfrak{R}$ is of the form $R=\bigsqcup_{i=1}^nR_i$, where $\bigsqcup$ denotes the disjoint union, and we let $m_R(\mathbf{x})= \sum_{i=1}^n m_{R_i}(\mathbf{x})$. The volume of an element $R\in \mathfrak{R}$ corresponding to the real valued function $A$ (or the $A$-volume of $R$) is
\[
    V_A(R)=\sum_{\mathbf{x}\in{\RRR\times\RRR}}A(\mathbf{x})m_R(\mathbf{x});
\]
clearly, this sum is finite. It is also clear that when specializing to rectangles, distributions and quasi-distributions this definition coincides with the usual definition of the volume. Index $A$ will be omitted if function $A$ is understood. As in \cite{OmSt} we observe that the multiplicity at a point and the volume at a real valued function $A$ are additive in rectangles:
\begin{description}
  \item[(a)] $m_{R_1\sqcup R_2}(\mathbf{x})=m_{R_1}(\mathbf{x})+m_{ R_2}(\mathbf{x})$;
  \item[(b)] $V_A(R_1\sqcup R_2)=V_A(R_1)+V_A( R_2)$.
\end{description}
Also following \cite{OmSt} we define a function $L$ of $R\in \mathfrak{R}$, and functions $P_M$ and $P_O$ of $\mathbf{x} \in\RRR^2$, all depending also on the real valued functions $A$ and $B$ such that $A\leqslant B$:
\begin{equation}\label{eq:LPmPo}
\begin{split}
   L^{(A,B)}(R) & =\sum_{\substack{\mathbf{y}\in{\RRR\times\RRR}\\m_R(\mathbf{y})>0}} B(\mathbf{y})m_R(\mathbf{y}) + \sum_{\substack{\mathbf{y}\in{\RRR\times\RRR}\\ m_R(\mathbf{y})<0}}A(\mathbf{y})m_R(\mathbf{y}) \\
   P_M^{(A,B)}(\mathbf{x})  & =\inf_{\substack{R\in\mathfrak{R}\\ m_R(\mathbf{x})>0}}\frac{L^{(A,B)}(R)}{m_R(\mathbf{x})} \quad\mbox{and}\quad P_O^{(A,B)}(\mathbf{x})= \inf_{\substack{R\in \mathfrak{R}\\ m_R(\mathbf{x})<0}} \frac{L^{(A,B)}(R)}{-m_R(\mathbf{x})},
\end{split}
\end{equation}
where infimum of an empty set is assumed equal to $+\infty$.

It is our first next goal to prove that there exists a distribution $F$ between two quasi-distributions $A\leqslant B$ (all having the same marginal distributions $F_X(x)$ and $F_Y(y)$) if and only if function $L^{(A,B)}$ is positive on all rectangles. On the way to that result, Theorem \ref{thm:L_positive}, we will first assume two conditions on the pair of real valued functions $(A,B)$ to be later specialized to a pair of quasi-distributions:
\begin{description}
  \item[(Q1)] $A\leqslant B$, and
  \item[(Q2)] $L^{(A,B)}(R)\geqslant0$ for all $R\in\mathfrak{R}$.
\end{description}
Here is a key lemma where we also need function $\gamma^{(A,B)} (\mathbf{x})= \min\{P_O^{(A,B)}(\mathbf{x}), B(\mathbf{x})- A(\mathbf{x})\}$ defined for $\mathbf{x}\in\DD$.

\begin{lemma}\label{lem:key}
  Let the pair of real valued functions $A\leqslant B$ satisfy Conditions $\mathbf{(Q1)}$ and $\mathbf{(Q2)}$, and let there exist an $\mathbf{x} \in\DD$ such that $t_0=\gamma^{(A,B)}(\mathbf{x})>0$. Then the pair of real valued functions $(A',B)$, where
\[
    A'(\mathbf{y})=\left\{
                     \begin{array}{ll}
                       A(\mathbf{x})+t, & \hbox{if $\mathbf{y}= \mathbf{x}$;} \\
                       A(\mathbf{y}), & \hbox{otherwise;}
                     \end{array}
                   \right.
\]
satisfies conditions $\mathbf{(Q1)}, \mathbf{(Q2)}$ for any $t, 0<t\leqslant  t_0$. If we choose $t=t_0$, then $\gamma^{(A',B)} (\mathbf{x})=0$.
\end{lemma}

The proof of this lemma goes exactly in the same way as the proof of \cite[Proposition 13]{OmSt} and will be omitted. Note that even if we started with quasi-distributions $(A,B)$ the function $A'$ would not be a quasi-distribution any more in general. Here is an extension of another result, i.e. \cite[Proposition 16]{OmSt}. The  original result is technically quite involved and goes through a number of stages. The extension of the proof is straightforward and will be omitted although it needs a careful reexamination.

\begin{proposition}\label{prop:L_positive}
  Let $A\leqslant B$ be discrete (bivariate) quasi-distributions with fixed marginal distributions. Then, there exists a discrete (bivariate) distribution $F$ with $A\leqslant F\leqslant B$ if and only if
\[
    L^{(A,B)}(R) \geqslant0
\]
for all $R\in \mathfrak{R}$.
\end{proposition}

In this proposition and in what follows the term ``fixed marginal distributions'' or ``fixed margins'' for short will mean that all bivariate distributions or quasi-distributions under consideration have the same $F_X$ and the same $F_Y$, while these univariate distributions need not be equal in general.

The proof of this proposition and of Theorem \ref{thm:sup inf} relies heavily on the following result which we call here a lemma although it is technically quite elaborate. 
Recall the notation  $\gamma^{(A,B)} (\mathbf{x})= \min\{P_O^{(A,B)}(\mathbf{x}), B(\mathbf{x})- A(\mathbf{x})\}$ introduced just before Lemma \ref{lem:key}. Note that, again, the proof of the lemma is a straightforward extension of \cite[Theorem 15]{OmSt} and will be omitted.

\begin{lemma}\label{lem:P O}
  If under the conditions $\mathbf{(Q1)},\mathbf{(Q2)}$ we have $\gamma^{(A,B)} (\mathbf{x})=0$ for all $\mathbf{x}\in \RRR\times \RRR$, then $V_A(R)\geqslant0$ for all rectangles $R\subseteq \RRR\times \RRR$.
\end{lemma}

Here is one of our main results which is a nontrivial extension of \cite[Theorem 17]{OmSt} and definitely needs to be proven. As a matter of fact, in order to do so we first need an additional lemma:

\begin{lemma}\label{lem:A F B}
  Let $A$ be a quasi-distribution and $F$ a distribution on $\RRR\times\RRR$ having fixed marginal distributions $F_X$ and $F_Y$. Let $\Delta_n\subseteq\Delta_{n+1}$ be a sequence of meshes for $n\in \mathds{N}$ whose union of corners $\mathcal{U}$ satisfies
  \begin{enumerate}[(1)]
    \item $\mathcal{U}$ is dense in $\RRR\times\RRR$,
    \item the first coordinates of members of $\mathcal{U}$ contain all jumps of $F_X$,
    \item the second coordinates of members of $\mathcal{U}$ contain all jumps of $F_Y$.
  \end{enumerate}
  If $F|_{\Delta_n}\leqslant A|_{\Delta_n}$ respectively $F|_{\Delta_n}\geqslant A|_{\Delta_n}$ for all $n\in\mathds{N}$, then $F\leqslant A$ respectively $F\geqslant A$. 
\end{lemma}

\begin{proof} We will only treat the case ``$\leqslant$'' since the other one goes in a similar way. First apply Theorem \ref{thm:sklar_quasi} on distribution $F$ and on quasi-distribution $A$ to obtain a copula $C$ and a quasi-copula $Q$ such that $F= C(F_X,F_Y)$ and $A=Q(F_X,F_Y)$. Clearly, it suffices to show that $C\leqslant Q$. Now, if a point $(x,y)\in\RRR \times\RRR$ is such that $x$ is a jump of $F_X$ and $y$ is a jump of $F_Y$, then this point belongs to a $\Delta_n$ for a large enough $n$ and the claim follows immediately. Next, if $x$ has this property and $y$ does not, then $F_Y$ is continuous at $y$ and there is a sequence of $y_k$ such that $(x,y_k)$ converges to $(x,y)$ and each of its terms belongs to a $\Delta_n$ for a large enough index $n$.
So, the pursued condition is satisfied at each of the terms and consequently in the limit. Similar argument applies if the roles of $x$ and $y$ are exchanged. And finally, when the point $(x,y)$ is such that both functions $F_X$ and $F_Y$ are continuous, we need to find a double sequence to conclude what is desired. Actually, we are done now since the values of $F$ and $A$ do not depend on the definitions of $C$ and $Q$ at other points.
\end{proof}

\begin{theorem}\label{thm:L_positive}
  Let $A\leqslant B$ be (bivariate) quasi-distributions with fixed marginal distributions $F_X(x)$ and $F_Y(y)$. Then, there exists a (bivariate) distribution $F$ with $A\leqslant F\leqslant B$ if and only if
\[
    L^{(A,B)}(R) \geqslant0
\]
for all $R\in \mathfrak{R}$. In this case $F$ has necessarily the same marginal distributions $F_X(x)$ and $F_Y(y)$.
\end{theorem}

\begin{proof}
  We first assume that there exists a distribution $F$ with $A\leqslant F\leqslant B$ and choose an $R\in \mathfrak{R}$. Note that $R$ is made of a finite number of rectangles that have a finite union of all possible corners. So, there exists a mesh $\Delta$ containing all these corners. Now, observe that $\langle A\rangle=A|_\Delta$ and $\langle B\rangle=B|_\Delta$ are discrete quasi-distributions, that $\langle F\rangle= F|_\Delta$ is a discrete distribution, and that $\langle A\rangle\leqslant \langle F\rangle\leqslant\langle B\rangle$. So, the desired conclusion follows by Proposition \ref{prop:L_positive}.

To get the inverse implication, assume that the condition of the theorem is fulfilled for all $R\in\mathfrak{R}$. Choose a sequence of meshes $\Delta_n\subseteq\Delta_{n+1}$ for $n\in\mathds{N}$ satisfying the three conditions of Lemma \ref{lem:A F B}. First, one may choose, say, $\Delta_n= \delta_{x^n} \times \delta_{y^n}$ determined by points
\[
    \delta_{x^n}=\delta_{y^n}=\left\{\frac{k}{2^n}\right\}_{k=-n2^n}^{n2^n}
\]
for $n\in\mathds{N}$ to satisfy Condition \emph{(1)}. We know that a monotone function has at most a countable set of discontinuities denoted by $\{x_i\}_i$. Let the new set $\delta_{x^n}$ be the union of the set defined above and the first (no more than) $n$ elements of the set $\{x_i\}_i$ to ensure Condition \emph{(2)}. We take care of Condition \emph{(3)} in a similar way.

For an $n\in\mathds{N}$ let $\mathfrak{R}_n$ be the set of disjoint unions of rectangles with corners in $\Delta_n$. Then, $A|_{\Delta_n}$ and $B|_{\Delta_n}$ satisfy the assumptions of Proposition \ref{prop:L_positive}, so that there exists a discrete distribution $F_n$ on $\Delta_n$ such that
\[
    A|_{\Delta_n}\leqslant F_n\leqslant B|_{\Delta_n}
\]
Now, extend the discrete distribution $F_n$ to a general distribution $\breve{F}_n= (F_n)^{BL}$ and denote its marginal distributions by $\breve{F}_X^n$ and $\breve{F}_Y^n$. The quasi-distributions $A$ and $B$ have fixed marginal distributions, so that inequality above implies  that $\breve{F}_X^n |_{\delta{x^n}}=F_X|_{\delta{x^n}}$ and $\breve{F}_Y^n |_{\delta{y^n}}=F_Y|_{\delta{y^n}}$. A simple consideration establishes that $\breve{F}_X^n$ converges to $F_X$ and that $\breve{F}_Y^n$ converges to $F_Y$. Finally, let $C_n$ be the copula with $\breve{F}_n=C_n(\breve{F}_X^n,\breve{F}_Y^n)$ and define $\widehat{F}_n=C_n({F}_X,{F}_Y)$. It follows that
\begin{equation}\label{eq:A F B}
  A|_{\Delta_n}\leqslant F_n= \widehat{F}_n|_{\Delta_n} 
  \leqslant B|_{\Delta_n}.
\end{equation}
Since the set of bivariate distributions having a fixed pair of marginal distributions is compact by Proposition \ref{prop:compact} there exists a subsequence $\widehat{F}_{n_k},k\in \mathds{N}$, converging uniformly to a distribution $F$. Clearly, $F|_{\Delta_n}=F_n$ for all $n\in\mathds{N}$ so that Lemma \ref{lem:A F B} concludes the proof of the theorem by Equation \eqref{eq:A F B}.
\end{proof}

In the following theorem we consider $A$ and $B$, a pair of bivariate quasi-distributions with fixed margins $F_X$ and $F_Y$, and denote
\[
    \mathcal{D}(A,B)=\{ F\in\mathfrak{D}^{F_X,F_Y}\,|\,A\leqslant F\leqslant B\}.
\]

\begin{theorem}\label{thm:sup inf}
  If $A\leqslant B$ are quasi-distributions with $\mathcal{D}(A,B) \neq\emptyset$, then\vskip3mm
  \begin{enumerate}[(a)]
  \item $\displaystyle
    B=\bigvee\mathcal{D}(A,B)\ \ \mbox{if and only if}\ \ B(\mathbf{x}) -A(\mathbf{x}) \leqslant P_O^{(A,B)}(\mathbf{x})$ for all $\mathbf{x}\in \RRR\times\RRR$.\\
  \item $\displaystyle
    A=\bigwedge\mathcal{D}(A,B)\ \ \mbox{if and only if}\ \ B(\mathbf{x})- A(\mathbf{x}) \leqslant P_M^{(A,B)}(\mathbf{x})$ for all $\mathbf{x}\in\RRR \times\RRR$.
  \end{enumerate}\vskip3mm
\end{theorem}

\begin{proof}
  Let us start by the proof of \emph{(a)}. Clearly, the condition of Theorem \ref{thm:L_positive} is fulfilled. Assume first that condition $B(\mathbf{x})-A(\mathbf{x}) \leqslant P_O^{(A,B)} (\mathbf{x})$ holds for all points $\mathbf{x}\in \RRR\times\RRR$ and that $\gamma^{(A,B)} (\mathbf{x})=B(\mathbf{x})- A(\mathbf{x})$ is strictly positive at a certain point $\mathbf{x}_0\in \RRR\times\RRR$. Choose a mesh, say $\Delta_n$ containing this point and apply Lemma \ref{lem:key} to replace $A$ by $A'$ with $A'(\mathbf{x}_0)= B(\textbf{x}_0)$ and $A'=A$ at all other points of the mesh. Clearly, $A\leqslant A'\leqslant B$, $L^{(A',B)} \geqslant0$ and $\gamma^{(A',B)}(\mathbf{x}_0)=0$ by Lemma \ref{lem:key}. We can repeat this procedure as long as there is a point in the mesh with a positive value of $\gamma$. If there is no point of the kind left in the mesh, the last corrected discrete quasi-distribution $A'$ is actually a discrete distribution by Lemma \ref{lem:P O} to be denoted by $F_n$. Since the mesh is finite we are done in a finite number of steps. The so obtained discrete distribution has the properties $A|_{\Delta_n}\leqslant F_n\leqslant B|_{\Delta_n}$ and at the same time $F_n(\mathbf{x}_0)= B(\mathbf{x}_0)$.
  Following the ideas from the proof of Theorem \ref{thm:L_positive} we continue by a sequence of meshes each contained in the next one whose union of corners is dense in $\RRR\times\RRR$ and by an according sequence of discrete distributions $F_n$ extended to a sequence of general distributions $(F_n)^{\mathrm{BL}}$. By going to a subsequence, if necessary, we may achieve a uniformly convergent sequence by Corollary \ref{prop:compact} and a limit distribution $F$ such that $A\leqslant F\leqslant B$ and at the same time, $F(\mathbf{x}_0)= B(\mathbf{x}_0)$, thus proving one direction of \emph{(a)}.

To get the proof of \emph{(a)} in the other direction assume that $B=\bigvee\mathcal{D}(A,B)$, choose $\mathbf{x} \in\RRR\times\RRR$, $\varepsilon>0$, and $F\in\mathcal{D}(A,B)$ such that
\[
    F(\mathbf{x})>B(\mathbf{x})-\varepsilon.
\]
It is clear that $P_O^{(A,B)}(\mathbf{x})\geqslant P_O^{(A,F)}(\mathbf{x})$. We want to show that
\begin{equation}\label{eq:last}
  P_O^{(A,F)}(\mathbf{x})\geqslant F(\mathbf{x})-A(\mathbf{x}).
\end{equation}
This will imply $P_O^{(A,B)}(\mathbf{x})>B(\mathbf{x})-A(\mathbf{x}) -\varepsilon$ and the desired conclusion will follow by the fact that $\varepsilon$ can be chosen arbitrarily small. In the proof of \eqref{eq:last} we first recall that
\[
    P_O^{(A,F)}(\mathbf{x})= \inf_{\substack{R\in \mathfrak{R}\\ m_R(\mathbf{x})<0}} \frac{L^{(A,F)}(R)}{-m_R(\mathbf{x})},
\]
where
\[
    L^{(A,F)}(R)=\sum_{m_R(\mathbf{y})>0}F(\mathbf{y})m_R(\mathbf{y}) +\sum_{m_R(\mathbf{y})<0}A(\mathbf{y})m_R(\mathbf{y}).
\]
We add to and subtract from these sums the sum of $F(\mathbf{y}) m_R(\mathbf{y})$ over $\mathbf{y}\in{\RRR\times\RRR}$ with $m_R(\mathbf{y})<0$ to get
\[
     L^{(A,F)}(R)= V_F(R) + \sum_{m_R(\mathbf{y})<0} (F(\mathbf{y}) - A(\mathbf{y})) (-m_R(\mathbf{y})) \geqslant (F(\mathbf{x}) - A(\mathbf{x})) (-m_R(\mathbf{x}))
\]
because all the summands of the above sum are nonnegative and they also contain the summand with $\mathbf{y}=\mathbf{x}$. This implies Equation \eqref{eq:last} thus finishing the proof of \emph{(a)}.

The proof of \emph{(b)} follows by taking the reflection (cf.\ the remark at the end of this proof) on the case \emph{(a)} \emph{mutatis mutandis}, i.e.\ once the necessary changes have been made. In particular, we apply the reflection on Equations \eqref{eq:LPmPo} and note that a reflection is exchanging the order on the lattice of quasi-distributions and by appropriately extending \cite[Lemma 6]{OmSt} also the main and opposite role of the corners of rectangles. So, we first get
\[
    L^{(B^\delta,A^\delta)}(\delta(R))=L^{(A,B)}(R)
\]
and then
\[
    P_O^{(B^\delta,A^\delta)}(\delta(\mathbf{x}))=P_M^{(A,B)}(\mathbf{x})  
    \quad\mbox{and}\quad P_M^{(B^\delta,A^\delta)}(\delta(\mathbf{x}))= P_O^{(A,B)}(\mathbf{x}). 
\]
The reflected \emph{(a)} becomes
\[
    B^\delta=\bigwedge\mathcal{D}(B^\delta,A^\delta)\ \ \mbox{\emph{if and only if}}\ \ A^\delta(\delta(\mathbf{x})) -B^\delta(\delta(\mathbf{x})) \leqslant P_M^{(B^\delta,A^\delta)}(\delta(\mathbf{x}))\ \ \mbox{\emph{for all}}\ \  x\in{\RRR\times\RRR}
\]
which is exactly \emph{(b)}. So, we are done by the first part of the proof.
\end{proof}

\textbf{Remark.} Recall that for any quasi-copula $Q$ the reflection $\sigma: x \mapsto 1-x$ induces a \emph{reflected quasi-copula}
$$Q^\sigma(x,y)=y-Q(1-x,y).$$
Similarly, for any quasi-distribution $F$ the reflection $\delta: x \mapsto -x$ induces a \emph{reflected quasi-distribution} by
$$F^\delta(x,y)=F_Y(y)-F(-x,y).$$
Observe that the so obtained quasi-distribution has the same second margin $F_Y(y)$, while the first margin $F_X(x)$ changes into $1-F_X(-x)$. In particular, in the standard probability approach, it is not c\`{a}dl\`{a}g any more, it is c\`{a}gl\`{a}d.\footnote{C\`{a}dl\`{a}g is a colourful French acronym (continue \`{a} droite, limit\'{e}s \`{a} gauche) to describe functions of a real variable which may have discontinuities, but are right continuous at every point, with a limit point to the left of every point (cf., say, \cite[p.\ 90]{Davi}); exchange left and right in this definition to get the meaning of c\`{a}gl\`{a}d. Possible English translations are RCLL (``right continuous with left limits''), or corlol (``continuous on (the) right, limit on (the) left''). } This is a phenomenon observed when exchanging a distribution function into a survival function. Observe that the same technique applies to a distribution which is just a special case of a quasi-distribution. In the case of ``the other'' reflection $\delta : y \mapsto 1-y$ applied to a quasi-distribution or distribution, we follow similar considerations. Note that, in particular, \cite[Lemma 6]{OmSt} has an immediate extension to quasi-distributions for either of the two reflections.

\section{The imprecise versions of the Sklar's theorem}\label{sec:impr_sklar}

\begin{theorem}[Imprecise Sklar's Theorem for Fixed Margins]\label{thm:sklar_fix}
  For quasi-distributions $A\leqslant B$ with fixed marginal distributions $F_X$ and $F_Y$ the following conditions are equivalent:
  \begin{enumerate}
    \item $B-A\leqslant P_O^{(A,B)}$ and $B-A \leqslant P_M^{(A,B)}$.
    \item $A=\inf\{F\in \mathfrak{D}^{F_X,F_Y}\,|\,A\leqslant F\leqslant B\}$ and $B=\sup\{F\in \mathfrak{D}^{F_X,F_Y}\,|\,A\leqslant F\leqslant B\}$.
    \item There exist quasi-copulas $P\leqslant Q$ satisfying
        \begin{equation}\label{eq:thm sklar}
            A=P(F_X,F_Y)\quad\mbox{respectively}\quad B=Q(F_X,F_Y)
        \end{equation}
        and
        \begin{equation}\label{eq:coherent}
          P=\inf\{C\ \mbox{copula}\,|\,P\leqslant C\leqslant Q\}\quad\mbox{and}\quad Q=\sup\{C\ \mbox{copula}\,|\,P\leqslant C\leqslant Q\}.
        \end{equation}
  \end{enumerate}
  Moreover, if $P\leqslant Q$ are any quasi-copulas satisfying \eqref{eq:coherent} and $F_X$ and $F_Y$ are any marginal distributions, then quasi-distributions defined by \eqref{eq:thm sklar} satisfy any and therefore all equivalent conditions (1)--(3).
\end{theorem}

\begin{proof} Observe that either Condition \emph{(1)} or \emph{(2)} implies automatically that $L^{(A,B)}\geqslant0$.
  Equivalence of \emph{(1)} and \emph{(2)} is then clearly given by Theorem \ref{thm:sup inf}. 
  Let $\mathcal{C}$ be the set of copulas $C$ such that for $F=C(F_X,F_Y)$ we have $A\leqslant F \leqslant B$ and define
  \[
    P=\inf\{C\in\mathcal{C}\}\quad\mbox{and}\quad Q=\sup\{C\in \mathcal{C}\},
  \]
  so that Equation 
  \eqref{eq:coherent} is satisfied as soon as $\mathcal{C}$ is nonempty. Now, if \emph{(2)} holds, than this set is nonempty by Theorem \ref{thm:sklar_quasi} and if \emph{(3)} holds it is nonempty by definition. Since for any point $\mathbf{x}\in\RRR \times\RRR$ there is a point $\mathbf{u}\in\II^2$ such that $F(\mathbf{x})=C(\mathbf{u})$, the pointwise infimum respectively supremum of the set given in \emph{(2)} is attained at a quasi-distribution $A$ respectively $B$ if and only if the pointwise infimum respectively supremum of the set $\mathcal{C}$ is attained at quasi-copula $P$ respectively $Q$ given by Relation \eqref{eq:thm sklar}. This proves the equivalence of \emph{(2)} and \emph{(3)}. It is now a simple exercise to show the second half of the theorem.
\end{proof}

Let $\mathcal{F}$ be a nonempty set of bivariate distributions and define the sets of its marginal distributions by
$\mathcal{F}_{X}=\set{F_X ~|~ F \in \mathcal{F}}$ and $\mathcal{F}_Y =\set{F_Y ~|~ F \in \mathcal{F}}$. In addition, let
$$
\mathcal{F}_{X,Y}=\set{(F_X,F_Y) ~|~ F \in \mathcal{F}} \subseteq \mathcal{F}_X \times \mathcal{F}_Y,
$$
where the elements of this set, i.e., pairs of possible marginal distributions will be denoted by $F_{X,Y}=(F_X,F_Y)$.
For every $F_{X,Y}\in\mathcal{F}_{X,Y}$ define
$\mathcal{F}^{F_{X,Y}}=\mathcal{F}\cap\mathfrak{D}^{F_X,F_Y}$ and note that each $\mathcal{F}^{F_{X,Y}}$ is a nonempty set of distributions with fixed margins $F_X$ and $F_Y$, and that
$$
\mathcal{F}=\bigcup_{F_{X,Y}\in\mathcal{F}_{X,Y}
}\mathcal{F}^{F_{X,Y}}.
$$
Furthermore, denote
$$
\underline{\mathcal{F}}^{{F_{X,Y}}}=\inf_{ F \in \mathcal{F}^{F_{X,Y}} } F 
\quad \textup{and} \quad \overline{\mathcal{F}}^{F_{X,Y}}=\sup_{ F \in \mathcal{F}^{F_{X,Y}} } F
.
$$

\begin{lemma}
  The set $\mathfrak{Q}^{F_X,F_Y}$ is closed under pointwise infima and suprema for any fixed pair $F_{X,Y}\in \mathcal{F}_{X,Y}$, so that $\underline{\mathcal{F}}^{{F_{X,Y}}}$ and $\overline{\mathcal{F}}^{F_{X,Y}}$ are quasi-distributions.
\end{lemma}

\textbf{Remark. } Although we prove here that the infimum and the supremum of a set of distributions with fixed margins are always quasi-distributions, it turns out that
$$
\underline{\mathcal{F}}=\inf_{F_{X,Y} \in \mathcal{F}_{X,Y}} \underline{\mathcal{F}}^{F_{X,Y}} \quad \textup{and} \quad \overline{\mathcal{F}}= \sup_{F_{X,Y} \in \mathcal{F}_{X,Y}} \overline{\mathcal{F}}^{F_{X,Y}}
$$
need not be quasi-distributions as shown in \cite[Example 1]{PeViMoMi2} and may even not be representable in the sense of Sklar's theorem as shown in \cite[Example 3.5]{PeViMoMi1} (cf.\ also \cite{MoMiPeVi}).

\begin{proof}[Proof of the lemma]
  The first claim of the lemma follows by a simple adjustment of the proof of  \cite[Theorem 6.2.5]{Nels}.
  Using Theorem \ref{thm:sklar_quasi} we find a set of copulas $\mathcal{C}^{F_{X,Y}}=\{C\,|\,F=C(F_X,F_Y)\ \mbox{for some}\ F\in \mathcal{F}^{F_{X,Y}}\}$ and let
$$
\underline{\mathcal{C}}=\inf_{C \in \mathcal{C}^{F_{X,Y}}} C \quad \textup{and} \quad \overline{\mathcal{C}}=\sup_{C \in \mathcal{C}^{F_{X,Y}}} C.
$$
The lemma follows by \cite[Theorem 6.2.5]{Nels} and the obvious fact that $\underline{\mathcal{F}}^{F_{X,Y}} = \underline{\mathcal{C}}(F_X,F_Y)$ and $\overline{\mathcal{F}}^{F_{X,Y}} = \overline{\mathcal{C}} (F_X,F_Y)$, so that they are quasi-distributions.
\end{proof}
\medskip

Following \cite{MoMiPeVi} (cf.\ also \cite{DiSaPlMeKl,OmSk}) we call a pair $(P, Q)$ of functions on $\DD$ an \emph{imprecise copula} if \textbf{(A)} they are grounded, \textbf{(B)} each of them has 1 as a neutral element, and
\begin{description}
  \item[(IC1)]  $P(\mathbf{a})+Q(\mathbf{c})-P(\mathbf{b})- P(\mathbf{d})\geqslant0$;
  \item[(IC2)] $Q(\mathbf{a})+P(\mathbf{c})-P(\mathbf{b})- P(\mathbf{d})\geqslant0$;
  \item[(IC3)] $Q(\mathbf{a})+Q(\mathbf{c})-Q(\mathbf{b})- P(\mathbf{d})\geqslant0$;
  \item[(IC4)] $Q(\mathbf{a})+Q(\mathbf{c})-P(\mathbf{b})- Q(\mathbf{d})\geqslant0$
\end{description}
for each rectangle $R\subseteq\DD$ defined by corners $\mathbf{a}, \mathbf{b}, \mathbf{c},$ and $\mathbf{d}$ in the standard way. It is known (cf.\ \cite{MoMiPeVi,DiSaPlMeKl,OmSk} for the general case and \cite{OmSt} for the discrete case) that for an imprecise copula $(P,Q)$ we have $P\leqslant Q$ and $P$ and $Q$ are quasi-copulas. The question whether every imprecise copula $(P,Q)$ satisfies Condition \eqref{eq:coherent} was proposed in \cite{MoMiPeVi} and answered in the negative in \cite{OmSt}. So, if an imprecise copula $(P,Q)$ does satisfy Condition \eqref{eq:coherent} we will call it a \emph{coherent imprecise copula} (cf.\ also \cite{OmSk}). The necessary and sufficient conditions for an imprecise copula to be coherent are also given in \cite[Theorem 19]{OmSt}. On the other hand, if a pair of quasi-copulas with $P\leqslant Q$ satisfies Condition \eqref{eq:coherent}, it is automatically an imprecise copula (cf.\ \cite{MoMiPeVi,DiSaPlMeKl,OmSt,OmSk}), so that we can take this condition as the primary definition of a coherent imprecise copula.

In what follows we need the notation
\[
    \mathcal{C}(P,Q)=\{C~|~C~\mbox{copula}~P\leqslant C\leqslant Q\}
\]
for any two quasi-copulas $P,Q$ with $P\leqslant Q$.

\medskip
\begin{theorem}[General Imprecise Sklar's Theorem]\label{thm:sklar1}~
\begin{enumerate}[(i)]
\item Let $\mathcal{F}$ be a nonempty set of bivariate distributions. Then there exists a family $\mathcal{C}^\mathcal{F}=\set{(P_{{F_{X,Y}}},Q_{{F_{X,Y}}}) ~|~ {{F_{X,Y}}\in {\mathcal{F}_{X,Y}}}}$ of coherent imprecise copulas with
$$
\underline{\mathcal{F}}^{{F_{X,Y}}}=P_{{F_{X,Y}}}(F_X, F_Y) \quad \textup{and} \quad \overline{\mathcal{F}}^{{F_{X,Y}}} =Q_{{F_{X,Y}}}(F_X, F_Y)
$$
for every ${{F_{X,Y}}\in {\mathcal{F}_{X,Y}}}$.
\item Let $\mathcal{F}_X$ and $\mathcal{F}_Y$ be two nonempty sets of univariate distributions and let $\mathcal{F}_{{X,Y}} \subseteq \mathcal{F}_X \times \mathcal{F}_Y$ be nonempty. Furthermore, let $\mathcal{C}=\{(P_{{F_{X,Y}}}, Q_{{F_{X,Y}}})\}_{{F_{X,Y}}}$ be a family of coherent imprecise copulas. Then
    $$
    \mathcal{F}=\set{C(F_X,F_Y) ~|~ {(F_X,F_Y)}\in \mathcal{F}_{{X,Y}}, ~C \in \mathcal{C}(P_{{F_{X,Y}}}, Q_{{F_{X,Y}}})}
    $$
is a nonempty set of bivariate distributions.
\end{enumerate}
\end{theorem}

\begin{proof}
$(i)$ For every ${{F_{X,Y}}\in \mathcal{F}_{X,Y}}$ the set $\mathcal{F}^{F_{X,Y}}$ is a set of distributions with fixed marginal distributions $F_X$ and $F_Y$. By Theorem \ref{thm:sklar_fix} there exists a coherent imprecise copula $(P_{F_{X,Y}},Q_{F_{X,Y}})$ such that
$$
\underline{\mathcal{F}}^{{F_{X,Y}}}(x,y)=P_{F_{X,Y}} (F_X(x), F_Y(y)) \quad \textup{and} \quad \overline{\mathcal{F}}^{{F_{X,Y}}} (x,y)=Q_{F_{X,Y}}(F_X(x),F_Y(y)).
$$
We let $\mathcal{C}=\set{(P_{F_{X,Y}},Q_{F_{X,Y}}) ~|~ {F_{X,Y}} \in \mathcal{F}_{X,Y}}$ to complete the proof.

$(ii)$ The fact that $\mathcal{F}$ is a set of bivariate distributions is clear because every $C$ is a copula and every pair $(F_X,F_Y)$ is a pair of margins.
For every coherent imprecise copula $(P,Q)$ the set of copulas $\mathcal{C}(P,Q)$ is nonempty by definition.
\end{proof}

\begin{example}
Let $F_1$ and $F_2$ be univariate distributions defined by
$$F_1(x)=\left\{\begin{array}{ll}
0, & x\leqslant \frac{1}{2},\\
2x-1, & \frac{1}{2}<x\leqslant 1,\\
1, & x>1,
\end{array}\right.
\quad \textup{and} \quad
F_2(x)=\left\{\begin{array}{ll}
0, & x\leqslant 0,\\
2x, & 0<x\leqslant \frac{1}{2},\\
1, & x>\frac{1}{2};
\end{array}\right.$$
and let $F$, $G$ and $H$ be bivariate distributions defined by
\begin{align*}
F(x,y) &=M(F_1(x),F_1(y)), \\
G(x,y) &=W(F_2(x),F_2(y)), \\
H(x,y) &=M(F_2(x),F_2(y)),
\end{align*}
where $M$ and $W$ are the upper and lower Fr\'{e}chet-Hoeffding bounds for copulas. Observe that $F(x,y)=0$ whenever either $x\leqslant \frac{1}{2}$ or $y\leqslant \frac{1}{2}$, and $G(x,y)= 1$ whenever both $x>\frac{1}{2}$ and $y>\frac{1}{2}$.
It follows that $F \leqslant G \leqslant H$. Hence
$$F=\inf\set{F,H} \quad \textup{and} \quad H=\sup\set{F,H}.$$
Both $F$ and $G$ are represented by the same copula $M$, while the (bivariate) 
$p$-box $(G,H)$, 
and hence also the $p$-box $(F,H)$, 
contains distributions $C(F_2(x),F_2(y))$ for any copula $C$. Consequently, a bivariate $p$-box in the sense of \cite{PeViMoMi2}, i.e., an interval of  bivariate distributions, cannot be represented by just one imprecise copula, i.e., an interval of copulas, unless having fixed margins (even if the lower and upper bounds are assumed to be quasi-distributions).
\end{example}

A bivariate $p$-box may be viewed as an extension of the notion of univariate $p$-box, where it usually means a set of all distribution functions $F$ lying pointwise between a smallest one $\underline{F}$ and a largest one $\overline{F}$, i.e., $\underline{F}(x)\leqslant F(x)\leqslant \overline{F}(x)$ for all $x\in\RRR$.
In view of the evidence presented in this paper and especially in this section so far we propose to study a slightly different notion to the (bivariate) $p$-box proposed in \cite{PeViMoMi2}. We believe that a more decisive notion is what we call a \emph{restricted (bivariate) $p$-box}; it is restricted in two ways: (1) it is made of quasi-distributions, and (2) the marginal univariate distributions of the whole set are two fixed (possibly distinct) univariate distributions $F_X$ and $F_Y$, a property we have been calling throughout the paper ``quasi-distributions with fixed margins''. Now, if $(A,B)$ is a restricted (bivariate) $p$-box, we call it \emph{coherent} if it has property \emph{(ii)} of Theorem \ref{thm:sklar_fix}. Observe that this condition is equivalent to having property \emph{(i)} of the same theorem. According to that theorem the coherent restricted (bivariate) $p$-boxes are in one-to-one correspondence with coherent imprecise copulas. Theorem \ref{thm:sklar1} may be seen as an extension of Theorem \ref{thm:sklar_fix} for families of coherent $p$-boxes. Using the notions introduced in this paragraph we can slightly reformulate Theorem \ref{thm:sklar1}.


\medskip
\begin{theorem}[Imprecise Sklar's theorem -- $p$-box approach]\label{thm:sklar2}~
\begin{enumerate}[(i)]
\item Let $\mathcal{F}$ be a nonempty set of bivariate distributions such that $\mathcal{F}^{F_{X,Y}}$ is a coherent restricted bivariate p-box for all $F_{X,Y}\in \mathcal{F}_{X,Y}$.
Then there exists a family $\mathcal{C}=\set{(P_{F_{X,Y}},Q_{F_{X,Y}}) ~|~ F_{X,Y}\in \mathcal{F}_{X,Y}}$ of coherent imprecise copulas such that
$$\mathcal{F}=\set{C(F_X(x),F_Y(y)) ~|~ C \in \mathcal{C} (P_{F_{X,Y}},Q_{F_{X,Y}}), F_{X,Y}\in \mathcal{F}_{X,Y}}.$$
\item Let $\mathcal{F}_X$ and $\mathcal{F}_Y$ be two univariate p-boxes and let $\mathcal{F}_{X,Y} \subseteq \mathcal{F}_1 \times \mathcal{F}_2$ be a nonempty subset. Furthermore, let $\mathcal{C}=\set{(P_{F_{X,Y}},Q_{F_{X,Y}}) ~|~ F_{X,Y}\in \mathcal{F}_{X,Y}}$ be a family of coherent imprecise copulas. Then
$$\mathcal{F}=\set{C(F_X(x),F_Y(y)) ~|~ C \in \mathcal{C}(P_{F_{X,Y}},Q_{F_{X,Y}}), F_{X,Y}\in \mathcal{F}_{X,Y}}$$
is a nonempty set of bivariate distributions such that $\mathcal{F}^{F_{X,Y}}$ is a coherent restricted bivariate $p$-box for all ${F_{X,Y}} \in \mathcal{F}_{X,Y}$.
\end{enumerate}
\end{theorem}

\begin{proof}
$(i)$ By Theorem \ref{thm:sklar_fix} for every $F_{X,Y}\in \mathcal{F}_{X,Y}$, there exists a coherent imprecise copula $(P_{F_{X,Y}},Q_{F_{X,Y}})$ such that
$$\mathcal{F}^{F_{X,Y}}=\set{C(F_X(x),F_Y(y)) ~|~ C \in \mathcal{C}(P_{F_{X,Y}},Q_{F_{X,Y}})}.$$
Taking $\mathcal{C}=\set{(P_{F_{X,Y}},Q_{F_{X,Y}}) ~|~ F_{X,Y}\in \mathcal{F}_{X,Y}}$ completes the proof.

$(ii)$ $\mathcal{F}$ is clearly a nonempty set of bivariate distributions. For every $F_{X,Y}\in \mathcal{F}_{X,Y}$ the set $\mathcal{F}^{F_{X,Y}}=\set{C(F_X(x),F_Y(y)) ~|~ C \in \mathcal{C}(P_{F_{X,Y}},Q_{F_{X,Y}})}$ is a coherent restricted bivariate p-box by Theorem~\ref{thm:sklar_fix}.
\end{proof}


\bibliographystyle{amsplain}

\begin{thebibliography}{00}

\bibitem{AlFrSc} C. Alsina, M. J. Frank, and B. Schweizer. Associative Functions: Triangular Norms and Copulas. World Scientific, Singapore 2006.

\bibitem{AlNeSc} C. Alsina, R. B. Nelsen, and B. Schweizer, \textsl{On the characterization of a class of binary operations on distribution functions} Statist.\ Probab.\ Lett.\ \textbf{17} (1993), 85--89.
    
\bibitem{Davi} J. Davidson, Stochastic limit theory, Oxford University Press, New York (1994). 

\bibitem{DeBa} B. De Baets: Quasi-copulas: A bridge between fuzzy set theory and probability theory. In: Integrated Uncertainty Management and Applications. Selected Papers Based on the Presentations at the 2010 International Symposium on Integrated Uncertainty Managment and Applications (IUM 2010) (V.-N. Huynh, Y. Nakamori, J. Lawry, and M. Inuiguchi, eds.), Ishikawa 2010, p.55. Springer, Berlin 2010.

\bibitem{DeBaJaDeMe} B. De Baets, S. Janssens, and H. De Meyer, \textsl{On the transitivity of a parametric family of cardinality-based similarity measures}, Internat. J. Approx. Reason. \textbf{50} (2009), 104--116.

\bibitem{DeScDeMeDeBa} B. De Schuymer, H. De Meyer, and B. De Baets, \textsl{Cycle-transitive comparison of independent random variables}, J. Multivariate Anal. \textbf{96} (2005), 352--373.

\bibitem{DiMoDeBa} S. D\'{i}az, S. Montes, and B. De Baets, \textsl{Transitivity bounds in additive fuzzy preference structures}, IEEE Trans. Fuzzy Systems \textbf{15} (2007), 275--286.


\bibitem{DiSaPlMeKl} M. Dibala, S. Saminger-Platz, R. Mesiar, and E. P. Klement, \textsl{Defects and transformations of quasi-copulas}, Kybernetika, \textbf{52} (2016), 848--865.

\bibitem{DuSe} F. Durante and C. Sempi. Principles of Copula Theory. CRC/Chapman \& Hall, Boca Raton, 2015.

\bibitem{GeQuMoRoLaSe} C. Genest, J. J. Quesada-Molina, J. A. Rodr\'{i}guez-Lallena, and C. Sempi, \textsl{A characterization of quasi-copulas}. J. Multivariate Anal. \textbf{69} (1999), 193--205.

\bibitem{HaMe} P. H\'{a}jek and R. Mesiar, \textsl{On copulas, quasicopulas and fuzzy logic}, Soft Computing \textbf{12} (2008), 1239--1243.

\bibitem{JaDeBaDeMe} S. Janssens, B. De Baets, and H. De Meyer, \textsl{[19] Bell-type inequalities for commutative quasi-copulas}, Fuzzy Sets and Systems \textbf{148} (2004), 263--278.


\bibitem{MoMiPeVi} I.\ Montes, E.\ Miranda, R.\ Pelessoni, P.\ Vicig, \textsl{Sklar's theorem in an imprecise setting}, Fuzzy Sets and Systems, \textbf{278} (2015), 48--66.


\bibitem{Nels} R.\ B.\ Nelsen, An introduction to copulas, 2nd edition, Springer-Verlag, New York (2006).

\bibitem{NeQuMoRoLaUbFl} R. B. Nelsen, J. J. Quesada-Molina, J. A. Rodr\'{i}guez-Lallena, and M. \'{U}beda-Flores, \textsl{Some new properties of quasi-copulas. In: Distributions with Given Marginals and Statistical Modelling (C.M. Cuadras, J. Fortiana, and J.A. Rodr\'{ı}guez-Lallena, edis.)}, Kluwer Academic Publishers, Dordrecht 2002, pp.187--194.

\bibitem{OmSt} M.\ Omladi\v{c}, N.\ Stopar, Final solution to the problem of relating a true copula to an imprecise copula, Accepted in Fuzzy sets and systems, DOI: 10.1016/j.fss.2019.07.002.

\bibitem{OmSk} M.  Omladi\v{c}, D. \v{S}kulj, \textsl{Constructing copulas from shock models with imprecise distributions}, preprint

\bibitem{PeViMoMi1} R. Pelessoni, P. Vicig, I. Montes, and E. Miranda, \textsl{Imprecise copulas and bivariate stochastic orders}. In: Proc. EUROFUSE 2013, Oviedo 2013, 217--224.

\bibitem{PeViMoMi2} R. Pelessoni, P. Vicig, I. Montes, E. Miranda. \textsl{Bivariate p-boxes}, International Journal of Uncertainty, Fuzziness and Knowledge-Based Systems, \textbf{24} (2016) 229--263.

\bibitem{QuMoSe} J. J. Quesada Molina, C. Sempi, \textsl{Discrete quasi-copulas}, Insurance: Mathematics and Economics \textbf{37} (2005) 27--41.

\bibitem{SaTuMe} E. Sainio, E. Turunen, and R. Mesiar, \textsl{A characterization of fuzzy implications generated by generalized quantifiers}. Fuzzy Sets and Systems \textbf{159} (2008), 491--499.

\bibitem{Skla} A.\ Sklar, Fonctions de r\'{e}partition \`{a} $n$ dimensions et leurs marges, Publ.\ Inst.\ Stat.\ Univ.\ Paris \textbf{8} (1959) 229--231.




\end{thebibliography}

\end{document}